\documentclass[12pt]{article}

\usepackage[dvips]{graphicx}
\usepackage{amsmath}
\usepackage{amssymb}
\usepackage{amsthm}

\newtheorem{theorem}{Theorem}

\newtheorem{lemma}{Lemma}

\theoremstyle{definition}

\newtheorem{remark}{Remark}
\newtheorem{example}{Example}

\newcommand{\balpha}{{\boldsymbol{\alpha}}}

\newcommand{\be}{\boldsymbol{e}}
\newcommand{\bx}{\boldsymbol{x}}
\newcommand{\bu}{\boldsymbol{u}}
\newcommand{\bv}{\boldsymbol{v}}
\newcommand{\bh}{\boldsymbol{h}}
\newcommand{\bt}{\boldsymbol{t}}
\newcommand{\bo}{\boldsymbol{0}}

\begin{document}

\title{On inequivalent factorizations of a cycle}

\author{G Berkolaiko$^1$, J M Harrison$^{2}$, M Novaes$^3$\\~\\
  \small $^1$ Department of Mathematics, Texas A\&M University,\\
  \small
  College Station, TX 77843-3368, USA\\
  \small
  $^2$ Department of Mathematics, Baylor University,
  Waco, TX 76798-7328, USA\\
  \small
  $^3$ Departamento de F\'isica, Universidade Federal de
  S\~ao Carlos, 13565-905 S\~ao Carlos, SP, Brazil}

\maketitle

\begin{abstract}
  We introduce a bijection between inequivalent minimal factorizations
  of the $n$-cycle $(1\,2\ldots n)$ into a product of smaller cycles
  of given length, on one side, and trees of a certain structure on
  the other.  We use this bijection to count the factorizations with a
  given number of different commuting factors that can appear in the
  first and in the last positions, a problem which has found
  applications in physics.  We also provide a necessary and sufficient
  condition for a set of cycles to be arrangeable into a product
  evaluating to $(1\,2\ldots n)$.
\end{abstract}

\section{Introduction}

Counting factorizations of a permutation into a product of cycles of
specified length is a problem with rich history, dating back at least
to Hurwitz \cite{Hur1891}, and with many important applications, in
particular in geometry (see e.g. \cite{EkeLanShaVai01}).  Our interest
in such problems is driven by applications encountered in physics,
namely semiclassical trajectory-based analysis of quantum transport in
chaotic systems \cite{MueHeuBraHaa07,BerHarNov08}.  The main
ingredient of this analysis is the existence of correlations between
sets of long trajectories connecting an input channel of the quantum
system to an output channel.  The trajectories organize themselves into
families, with the elements of a family differing among themselves
only by their behavior in small regions (see Fig.~\ref{fig:corr}) in
which some of them have crossings while others have anti-crossings.
Enumerating possible configurations of crossing regions and their
inter-connectivity is a question of combinatorial nature.  In a
certain special case it was found in \cite{BerHarNov08} to be
connected to the {\em inequivalent} {\em minimal} factorizations of
the $n$-cycle $(1\,2\ldots n)$ into a product of smaller cycles.

\begin{figure}[t]
\centerline{\includegraphics[clip,scale=1]{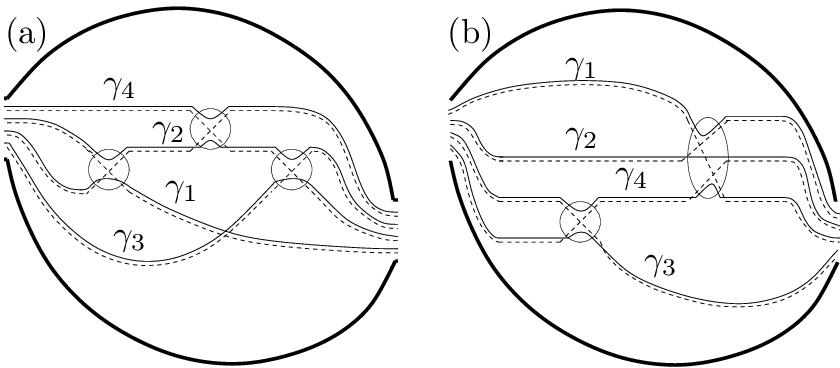}}
\caption{Two schematic examples of correlated sets of classical
  trajectories through a quantum system.  One set of trajectories is
  represented by solid lines, while the other set is drawn in dashed
  lines.  The circles mark the crossing regions, where the
  trajectories from different sets differ significantly: dashed lines
  cross while solid lines narrowly avoid crossings.}
\label{fig:corr}
\end{figure}

The base for our results is a simple and highly pictorial bijection between
said factorizations and plane trees.  This bijection allows us to recover
easily some already known results and to answer new questions about the
structure of the set of factorizations.  To be more specific we need to
introduce some notation.

Let $\sigma_m \cdots \sigma_1$ be a factorization of the cycle $(1\,2\ldots n)$
into a product of smaller cycles. By convention, the first entry of a cycle is
always its smallest element. We say the factorization is of type $\balpha$ if
among $\{\sigma_j\}$ there are exactly $\alpha_2$ 2-cycles (transpositions),
$\alpha_3$ 3-cycles and so on. Let us define 
\begin{equation}
  |\balpha|=\sum_{j\ge 2} \alpha_j, \quad \langle \balpha\rangle=\sum_{j\ge 2}
  (j-1)\alpha_j.
\end{equation} 
The quantity $\balpha$ satisfies
\begin{equation}
  \label{rank} 
  \langle \balpha\rangle\ge n-1.
\end{equation} 
If the above relation becomes equality, the factorization is called
{\em minimal}. We only consider minimal factorizations.

If two factorizations differ only in the order of commuting factors, they are
said to be {\em equivalent}.  An example of two equivalent factorizations is
\begin{equation}
  \label{eq:fact_example}
  (1\,2\,3\,4) = (3\,4)(1\,2)(2\,4) = (1\,2)(3\,4)(2\,4).
\end{equation}
From now on we will refer to equivalence classes of factorizations
simply as factorizations, unless the distinction is of particular
importance. In Theorem 1 we establish a bijection between
factorizations of type $\balpha$ and plane trees with vertex degrees
determined by $\balpha$.  In turn, the trees have been enumerated by
Erd\'elyi and Etherington \cite{ErdEth41} (see also Tutte \cite{Tut64}
and Stanley \cite{StanleyV2}, theorem 5.3.10). We are interested in
counting the factorizations of type $\balpha$. This number will be
denoted\footnote{We use the same notation as \cite{Irv06}; there $H$
  was counting \emph{all} factorizations} by $\widetilde{H}(\balpha)$.
In Theorem 2 we give an equation for its generating function and its
relation to Catalan numbers.

Given an equivalence class of factorizations of the form $\sigma_m
\cdots \sigma_1$, we refer to the number of different cycles that can
appear in the position $\sigma_m$, the {\em number of heads\/} of the
factorization.  Similarly, the {\em number of tails\/} is the number
of cycles that can appear in the position $\sigma_1$.  For example,
the factorization in (\ref{eq:fact_example}) has 2 heads
(transpositions $(1\,2)$ and $(3\,4)$) and 1 tail.  In Theorem 3 we
derive a generating function for the number of inequivalent minimal
factorizations with the specified number of heads and tails, denoted
by $\widetilde{H}_{\bh,\bt}(\balpha)$.  The vectors $\bh=(h_2, h_3,
\ldots)$ and $\bt=(t_2, t_3, \ldots)$ characterize the number of heads
and tails.  Namely, $h_j$ is the number of $j$-cycle heads and $t_j$
is the number of $j$-cycle tails.  The quantity
$\widetilde{H}_{\bh,\bt}(\balpha)$ is of importance in applications to
quantum chaotic transport \cite{BerHarNov08}.  Looking again at
Fig.~\ref{fig:corr}, the vector $\bh$ (corresp.\ $\bt$) counts the
number of crossings that can happen close to the left (corresp.\
right) opening of the system.  For example, $t_2=1$ on the diagram (a)
while $t_2=0$ on the diagram (b), since the $3$-crossing prevents the
$2$-crossing from getting close to the right opening.

Finally, we will give a complete characterization (necessary and
sufficient conditions) for a set $\{\sigma_j\}$ of cycles to give rise
to a minimal factorization of the $n$-cycle $(1\,2\ldots n)$.  This
characterization is given in Theorem~\ref{thm:struct_fact} (the
corresponding result for the factorization into transpositions can be
traced back to Eden and Sch{\"u}tzenberger \cite{EdeSch_mtamkik62}).
Here we only mention one of its corollaries: a factorization
equivalence class is completely determined by the factors.  In other
words, two factorizations composed of the same factors are equivalent.

Some of the results discussed in this paper are already known, although they
have been derived using different methods.  Namely, the number of
inequivalent factorizations into a product of transpositions (i.e.
$\balpha=(n-1, \boldsymbol{0})$) has been obtained by Eidswick \cite{Eid89} and
Longyear \cite{Lon89}.  Springer \cite{Spr96} derived a formula for
$\widetilde{H}(\balpha)$ using a different bijection to trees of the same type.
Irving \cite{Irv06} reproduced the result of Springer using more general
machinery involving cacti. The novelty of our approach is in the type (and
simplicity!) of the bijection used.  Being very visual, our bijection allows us
to obtain answers to new questions, namely to count factorizations with
specified number of heads and tails, which proved to be invaluable in
applications \cite{BerHarNov08,KuiWalPet_prl10,BerKui_jpa10}.

Of other related results we would like to mention Hurwitz
\cite{Hur1891} who suggested a formula for the number of minimal
transitive factorizations (counting equivalent factorizations as
different) of a general permutation into a product of 2-cycles.  A
factorization is called {\em transitive} if the group generated by the
factors $\sigma_1,\ldots, \sigma_m$ acts transitively on the set
$1,\ldots,n$.  However, Hurwitz gave only a sketch of a proof and his
paper was largely unknown to the combinatorialists.  For a special
case of factorizations of the $n$-cycle, the formula was (re-)derived
by D\'enes \cite{Den59}, with alternative proofs given by Lossers
\cite{Los86}, Moszkowski \cite{Mos89}, Goulden and Pepper
\cite{GouPep93}.  For general permutations, Strehl \cite{Str96}
reconstructed the original proof of Hurwitz, filling in the gaps,
while Goulden and Jackson \cite{GouJac97} gave an independent proof.
Generalizations of Hurwitz formula to factorizations into more general
cycles were considered in Goulden and Jackson \cite{GouJac00} and
Irving \cite{Irv06}.  Finally, {\em inequivalent} minimal transitive
factorizations of a permutation consisting of $m=2$ cycles have been
counted in Goulden, Jackson and Latour \cite{GouJacLat01} (into
transpositions) and in Irving \cite{Irv06} (into general cycles).  For
permutations with $m=3$ and $4$ cycles formulas have been found
\cite{BerIrv_prep} using the technique presented in the current
manuscript but generalizations to larger $m$ appear to be difficult.

\section{Visualizing a product of cycles}
\label{sec:viz}

A particularly nice way to visualize a product of transpositions was
suggested in \cite{Gar59} (see also \cite{Bog08}).  A permutation from
$S_n$ is represented as $n$ labeled horizontal lines with several
vertical lines (``shuttles'') connecting some pairs of the horizontal
lines, see Fig.~\ref{fig:diagram}.  The right and left ends of a line
$k$ are labeled with $t_k$ (for ``tail'') and $h_k$ (for ``head'')
correspondingly.  For every horizontal line, start at the right and
trace the line to the left.  Wherever an end of a shuttle is
encountered, trace this shuttle vertically till its other end and then
resume going to the left (towards $h$).  Continue in this manner until
you reach the left end of one of the horizontal lines.  It is clear
that the mapping ``right ends to left ends'' thus described is
invertible and therefore one-to-one.

\begin{figure}[t]
  \centering
  \includegraphics[scale=0.7]{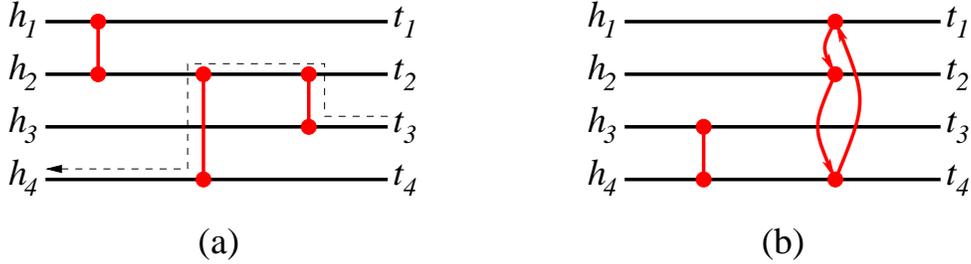}
  \caption{Visualizing a product of transpositions using a ``shuttle
    diagram''.  Each term $(k_1\ k_2)$ in the product corresponds to a
    vertical edge (``shuttle'') connecting lines $k_1$ and $k_2$.  The
    lines are ordered in the same way as the terms in the product.
    Part (a) depicts the product $(1\, 2)(2\, 4)(2\, 3)$ and finding
    the image of $3$ under the resulting permutation (dashed line).
    Part (b) is the representation of $(3\, 4)(1\,2\,4)$, where the
    longer cycle is represented by a directed shuttle.  Note that
    these two shuttle diagrams correspond to the trajectory
    reconnections depicted on Fig.~\ref{fig:corr}: the solid lines
    connecting $h_j$ to $t_j$ correspond to the trajectories drawn in
    solid lines and labelled $\gamma_j$ in Fig.~\ref{fig:corr}.  The
    dashed lines of Fig.~\ref{fig:corr} correspond to the paths from
    $t_j$ to $h_{j+1}$ via the shuttles.}
  \label{fig:diagram}
\end{figure}

In this construction, a shuttle connecting lines $k_1$ and $k_2$ represents
the transposition $(k_1\, k_2)$.  The transpositions are ordered in the same
way as shuttles: right to left.  If the two neighboring transpositions $(k_1\,
k_2)$ and $(k_3\, k_4)$ commute (if and only if all four $k_j$ are distinct),
the corresponding shuttles can be swapped around without affecting the
dynamics.  We can view the resulting diagram as a graph (with horizontal and
vertical edges).  If the diagram represents a factorization of an $n$-cycle,
the graph is connected.  By counting vertices and edges, one concludes that if
a factorization is {\em minimal}, the resulting graph is a tree.

Suppose now that the diagram represents a minimal factorization of the cycle
$(1\, 2 \ldots n)$.  In addition to the right-to-left motion described above we
define the left-to-right motion as going horizontally, ignoring the shuttles.
Then, starting at $t_1$ and going left we arrive to $h_2$.  Going right from
there we arrive to $t_2$ and from there, to $h_3$.  Continuing in this fashion,
we obtain a closed walk with several important features.  It visits the
vertices $t_1, h_2, t_2, \ldots, h_n, t_n, h_1$ in this sequence.  It traverses
each edge of the graph exactly twice: once in each direction (this follows, for
example, from the invertibility of the motion).  Since the graph is a tree, we
conclude that it goes from one vertex to the next one along the shortest
possible route.  This walk traversing the entire tree will play an important
role in what follows.

\begin{figure}[t]
  \centering
  \includegraphics[scale=0.7]{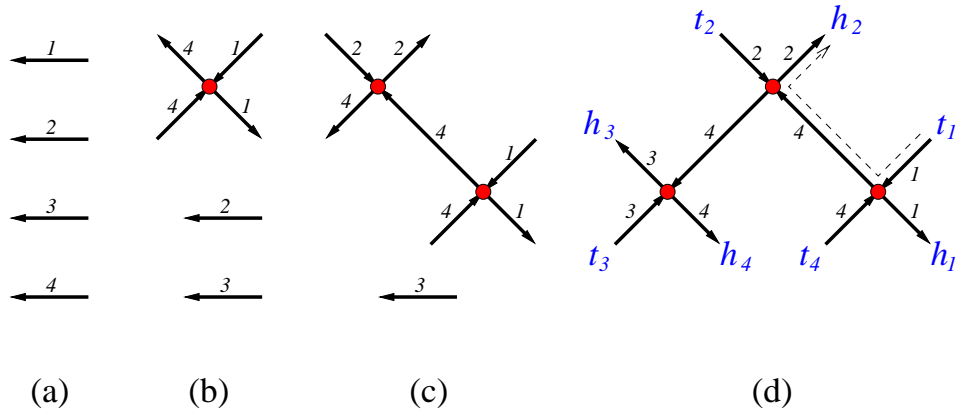}
  \caption{Visualizing a product of transpositions as a directed
    graph.  Depicted are the steps in constructing the graph corresponding to
    the product $(3\, 4)(2\, 4)(1\, 4)$.  To read off the image of $k$ under
    the resulting permutation we start at $t_k$ and follow the directions of
    the edges, choosing the next edge in the counterclockwise order at each
    vertex, until arriving to $h_{\pi(k)}$.  The path traced starting with
    $t_1$ is illustrated by the dashed line in (d).}
  \label{fig:directed}
\end{figure}

Another way to visualize a product of transpositions as a directed plane graph
is illustrated on Fig.~\ref{fig:directed}.  We start with $n$ disjoint
directed edges labeled $1$ to $n$.  For a product $\pi = (k_{j-1}\,
k_j)\cdots(k_3\, k_4)(k_1\, k_2)$, we start by joining the heads of the edges
labeled $k_1$ and $k_2$ at a new vertex and add two more {\em outgoing\/}
edges also labeled $k_1$ and $k_2$.  We arrange them around the vertex so
that, going counter-clockwise, the outgoing edge $k_1$ is followed by the
incoming $k_1$, then by the outgoing $k_2$ and, finally, by the incoming
$k_2$.  At this and all later stages of the procedure for each $k=1,\ldots,n$
there is exactly one ``free'' head of an edge labeled $k$, and one free tail
of possibly different edge also labeled $k$.  We now repeat the procedure
for the transposition $(k_3\, k_4)$, joining free heads of edges marked $k_3$
and $k_4$, adding new outgoing edges to new vertex and ordering the edges in
the similar fashion: outgoing $k_3$, incoming $k_3$, outgoing $k_4$ and
incoming $k_4$.  On Fig.~\ref{fig:directed}(d) we identified the free heads
and tails by labeling them with $h_j$ and $t_j$ correspondingly.

The resulting graph is closely related to the diagrams described earlier.
Namely, the graph is obtained from the diagram by shrinking the shuttle edges
and re-ordering the edges at the newly merged vertices, see
Fig.~\ref{fig:shrinking}(a).  Moreover, the ordering of edges has been
designed so that, to determine the image of $k$ under the product permutation
$\pi$, one would start at $t_k$ and travel along the direction of the edges,
at each vertex taking the next edge in the counterclockwise order, finally
arriving to $h_{\pi(k)}$.  This is illustrated by the dashed line on
Fig.~\ref{fig:directed}(d).  Starting at $h_k$ and going {\em against\/} the
direction of the edges, taking the next counterclockwise edge at each vertex,
will get one to $t_k$.

\begin{figure}[t]
  \centering
  \includegraphics[scale=0.7]{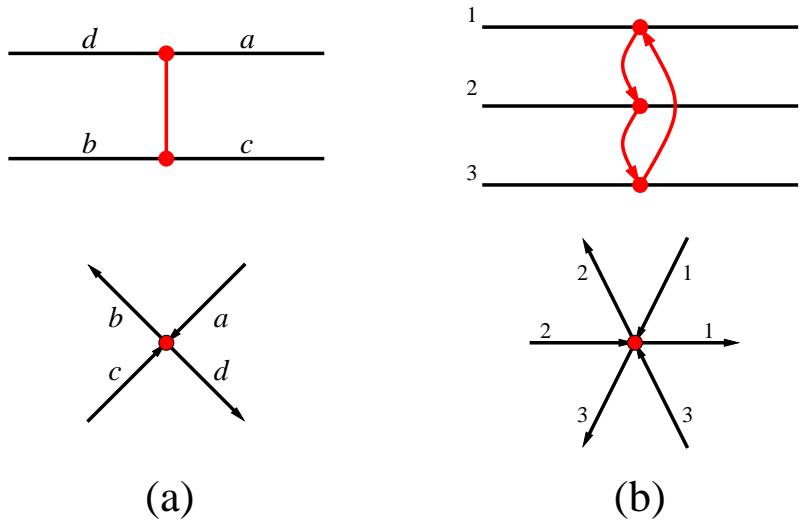}
  \caption{(a) Transforming the diagram representation of a product of
    transpositions into a directed graph representation.  The ``shuttle'' edge
    is shrunk and its end-vertices are merged.  The edges on the left are
    labeled in the order they are traversed by the walk $t_1, h_2, t_2,
    \ldots, h_n, t_n, h_1, t_1$. (b) Visualization of the cycle $(1\,2\,3)$ as
    a ``shuttle diagram'' and as a directed graph.}
\label{fig:shrinking}
\end{figure}

Thus, if $\pi$ is the cycle $(1\,2 \ldots n)$, the corresponding graph
is a tree with $n-1$ vertices of total degree 4 (henceforth called
{\em internal vertices\/}), $2n$ vertices of degree 1 (henceforth
called {\em leaves\/}) and $3n-2$ edges.  The walk $t_1, h_2, t_2,
\ldots, h_n, t_n, h_1, t_1$, discussed in the context of diagrams, now
circumnavigates the entire tree in the counter-clockwise direction.  As
before, it traverses each edge exactly once in each direction.  The
leaves of the tree are thus marked $h_1, t_1, h_2, t_2,
\ldots, h_n, t_n$ going counter-clockwise, see Fig.~\ref{fig:directed}.

The generalization of this construction from a product of
transpositions to a product of general cycles is straightforward.  For
a $m$-cycle $(k_1\, k_2 \ldots k_m)$, the corresponding shuttle is
realized as $m$ directed edges indicating transitions from horizontal
line $k_j$ to horizontal line $k_{j+1}$.  In the directed graph
visualization, the cycle corresponds to a vertex of total degree $2m$,
with outgoing edge marked $k_1$ followed by the incoming edge $k_1$,
then by outgoing edge $k_2$ and so on.  We illustrate this in
Fig.~\ref{fig:diagram}(b) using the cycle $(1\, 2\, 4)$ as an example.

\section{Main Results}
\label{sec:main}

As described in section~\ref{sec:viz}, a factorization (up to
equivalence) of the $n$-cycle into $n-1$ transpositions is naturally
represented as a plane tree with $n-1$ internal vertices of total degree 4,
$2n$ leaves of degree 1 and $3n-2$ edges.  If we designate the leaf $h_1$ as the
root of the tree, the labeling of all other leaves and the directions of edges
can be reconstructed uniquely.  This representation of a factorization as an
undirected rooted plane tree turns out to be a bijection.

\begin{figure}[t]
  \centering
  \includegraphics[scale=0.6]{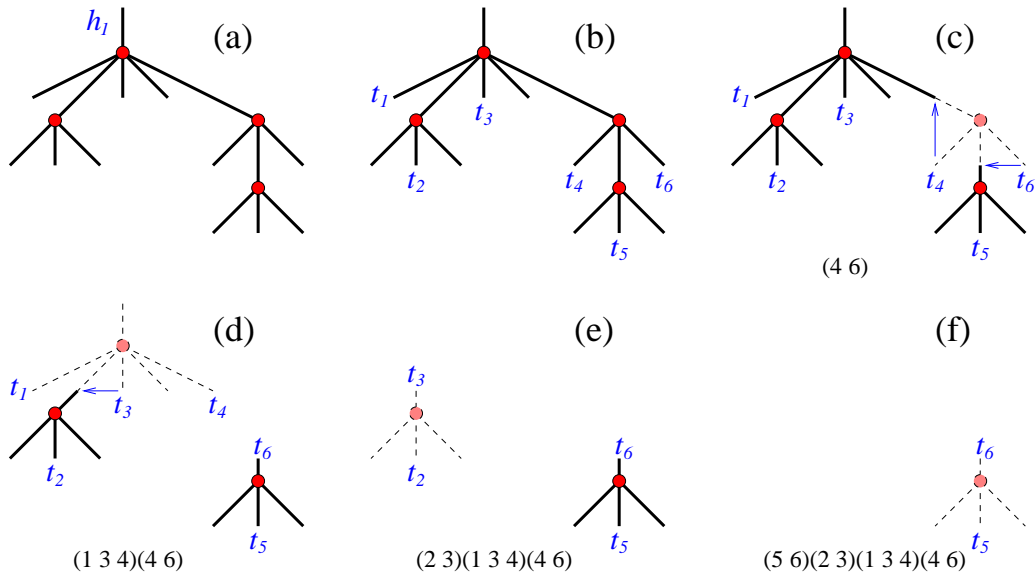}
  \caption{Reconstructing the factorization from an undirected rooted plane
    tree.  Note that at step (d) and (e) one could choose to remove the
    vertex adjacent to leaves $t_4$ and $t_6$ instead.}
  \label{fig:tree_to_fact}
\end{figure}

\begin{theorem}
  \label{thm:bijection}
  Inequivalent minimal factorizations of type $\balpha$ of the $n$-cycle $(1\,2\ldots n)$ are in
  one-to-one correspondence with undirected rooted plane trees having
  $\alpha_j$ vertices of degree $2j$ and $2n$ leaves of degree 1.
\end{theorem}

\begin{proof}
  The mapping of factorization equivalence classes to trees, described
  in section~\ref{sec:viz} is well defined.  Indeed, the construction
  steps corresponding to commuting factors also commute.  We need to
  show that this mapping is invertible and onto.

  The mapping can be inverted by taking the following steps (see
  Fig.~\ref{fig:tree_to_fact} for an example):
  \begin{enumerate}
  \item Label the leaves of the tree with $h_1, t_1, h_2, \ldots, t_n$ starting
    with the root and going counterclockwise, Fig.~\ref{fig:tree_to_fact}(b)
    (to avoid clutter we will omit the $h$-labels).
  \item \label{itm:start_loop} For some value of $j$, choose a vertex with
      degree $2j$ which has $j$ $t$-leaves adjacent to it.  Such a vertex
      exists by pigeonhole principle.  The indices of the $t$-leaves give
      the next (in the right to left order) factor in the expansion,
      Fig.~\ref{fig:tree_to_fact}(c).
  \item Remove the vertex.  The edges connecting the vertex to leaves are
    removed entirely.  The edges connecting the vertex to other vertices, if
    any, are cut in half.  This creates one or more new leaves and their labels
    are inherited from the leaves neighboring them in the counterclockwise
    direction, Fig.~\ref{fig:tree_to_fact}(c).
  \item Repeat from step \ref{itm:start_loop},
    Fig.~\ref{fig:tree_to_fact}(d)-(f).
  \end{enumerate}

Notice that the number of choices one has when first running step $2$
corresponds to the total number of tails.

  To verify that the mapping is onto we have to check that the above inversion
  applied to any tree produces a factorization of the $n$-cycle $(1\,2\ldots
  n)$.  To this end we observe that the deletion-relabeling process
  coupled with the application of the cycles read at step \ref{itm:start_loop}
  transports an object initially at $t_j$ to the leaf $h_{j+1}$ for all $j$
  (assuming the convention $n+1\equiv 1$).
\end{proof}

Before we proceed to counting trees, we would like to present a short
corollary of the above theorem.

\begin{lemma}
  \label{lem:increasing}
  Let $\sigma=(s_1\ \ldots\ s_{|\sigma|})$ be a cycle in a
  factorization of the $n$-cycle $(1\,2\ldots n)$.  Then $\sigma$ is
  increasing: $s_1 < s_2 < \ldots < s_{|\sigma|}$.
\end{lemma}

\begin{proof}
  When reading the factorization off the tree as described in the proof of
  Theorem~\ref{thm:bijection}, the labels are assigned initially to the leaves
  of the tree in the counterclockwise order.  The operation of removing a
  vertex and inheriting the labels preserves this ordering of the labels.  If
  a new connected component is created by the removal operation, its labels
  are also ordered counterclockwise.  Thus, when a vertex $\sigma$ is removed,
  the labels of its leaves, $t_{s_1}, \ldots, t_{s_{|\sigma|}}$ satisfy $s_1 <
  s_2 <\ldots < s_{|\sigma|}$ (provided the starting index $s_1$ is chosen
  appropriately).  Thus each factor read off the tree is an increasing cycle.
\end{proof}

\begin{theorem}
  \label{thm:gf_general}
  The generating function of the number
  $\widetilde{H}(\balpha)$, defined by
  \begin{equation}
    \label{eq:gen_fun_simple}
    \xi(\bx) = \sum_{\balpha} \widetilde{H}(\balpha) x_2^{\alpha_2}
    x_3^{\alpha_3} \cdots,
  \end{equation}
  where the sum over $\balpha$ is unrestricted, satisfies the
  recurrence relation
  \begin{equation}
    \label{eq:gf_general}
    \xi(\bx) = 1 + x_2 \xi^3(\bx) + x_3 \xi^5(\bx) + \ldots
  \end{equation}
  It follows that 
  \begin{equation}
    \sum_{\balpha:\langle \balpha\rangle=n}
    (-1)^{|\balpha|+n}\widetilde{H}(\balpha)=\frac{1}{n+1}\binom{2n}{n}
    = c_n,
  \end{equation}
  where $c_n$ is the $n$-th Catalan number.
\end{theorem}

The above statement is a simple consequence of the bijection between
factorizations and trees and the known results enumerating the trees,
see Erd\'elyi and Etherington \cite{ErdEth41}, Tutte \cite{Tut64} or
Stanley \cite[Theorem 5.3.10]{StanleyV2}.  We will give a short proof
in section~\ref{sec:counting} to introduce the methods used in the
next result. 

For factorizations with specified numbers of heads and tails we have
\begin{theorem}
  \label{thm:gf_heads_tails}
  Let $g(\bx,\bv,\bu)$ be the generating function of the number
  $\widetilde{H}_{\bh,\bt}(\balpha)$ of inequivalent minimal factorizations of
  the $n$-cycle $(1\,2\ldots n)$ of type $\balpha$ with specified number of heads and
  tails, defined by
  \begin{equation*}
    g(\bx,\bv,\bu) =  \sum_{\balpha}
    \sum_{\bh=(0,0,\ldots)}^{\balpha} \sum_{\bt=(0,0,\ldots)}^{\balpha}
    \widetilde{H}_{\bh,\bt}(\balpha)
    x_2^{\alpha_2} u_2^{h_2} v_2^{t_2}
    x_3^{\alpha_3} u_3^{h_3} v_3^{t_3} \cdots
  \end{equation*}
  Then $g(\bx,\bv,\bu)$ can be found as
  \begin{align}
    g &= f - \sum_{n\geq2} x_n (1-u_n) f^n \label{eq:rec_full_unsym}\\
    \label{eq:rec_full_sym}
    &= f\hat{f} -  \sum_{n\geq2} x_n \left(f\hat{f}\right)^n,
  \end{align}
  where $f$ satisfies the recursion relation
  \begin{equation}
    \label{eq:reduced_rec}
    f(\bx,\bv,\bu) = 1
    + \sum_{n\geq2} x_n \left(f^n - 1 + v_n\right) \hat{f}^{n-1}
  \end{equation}
  and $\hat{f}$ is obtained from $f$ by exchanging the roles of $u$
  and $v$,
  \begin{equation*}
    \hat{f}(\bx,\bv,\bu) = f(\bx,\bu,\bv).
  \end{equation*}
\end{theorem}

Moving on to the characterization of all possible sets of factors, we
remind the reader that a cycle $(s_1\,s_2\ldots s_{|\sigma|})$ is
called {\em increasing\/} if $s_1 < s_2 <\ldots < s_{|\sigma|}$.  Here
by $|\sigma|$ we denote the size of the cycle $\sigma$.

Another natural way to visualize factorizations of cycles is to draw
the factors on a circle.  We start by arranging the numbers $1,\ldots,
n$ on the circle in the anti-clockwise direction.  For a cycle
$(s_1\,s_2\ldots s_{|\sigma|})$ we will draw curves \emph{inside} the
circle connecting $s_j$ to $s_{j+1}$ for $j=1,\ldots,|\sigma|$, with
the last curve connecting $s_{|\sigma|}$ to $s_1$.  We draw the curves
without intersections (apart from at the vertices $s_j$), which is
possible if and only if the cycle is increasing or decreasing.  For
uniformity, in case of a transposition $\sigma = (s_1, s_2)$, we draw
two curves.  An example with all cycles in a factorization drawn on a
circle is shown on Fig.~\ref{fig:cactus}.

\begin{figure}[t]
  \centering
  \includegraphics[scale=0.6]{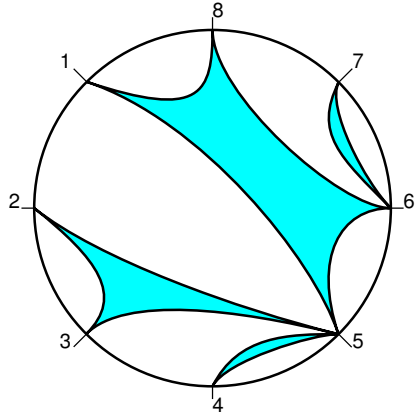}
  \caption{Cycles $(1\,5\,6\,8)$, $(2\,3\,5)$, $(4\,5)$ and $(6\,7)$
    drawn on a circle.  The cycles satisfy conditions of
    Theorem~\ref{thm:struct_fact} and can be arranged into a
    factorization $(1\,2\,3\,4\,5\,6\,7\,8) =
    (4\,5)(2\,3\,5)(1\,5\,6\,8)(6\,7)$.}
  \label{fig:cactus}
\end{figure}

\begin{theorem}
  \label{thm:struct_fact}
  The cycles $\sigma_1,\ \sigma_2,\ldots,\ \sigma_m$ can be arranged
  into a factorization of the $n$-cycle $(1\,2\ldots n)$ if and only
  if the following conditions hold
  \begin{enumerate}
  \item \label{itm:cover_all} any number $j=1,\ldots, n$ belongs to at
    least one of the cycles,
  \item \label{itm:increasing} all cycles are increasing,
  \item \label{itm:no_intersect} the cycles can be drawn on a circle
    without intersecting themselves and one another (apart from at the
    vertices $j=1,\ldots,n$ on the circle),
  \item \label{itm:simply_conn} the closed union of the resulting
    curvilinear polygons is simply connected.
  \end{enumerate}
  All factorizations formed out of these cycles are equivalent.
\end{theorem}

\begin{remark}
  Either requirement number~\ref{itm:cover_all} or requirement number~\ref{itm:simply_conn}
  can be substituted by the condition that
  \begin{equation}
    \label{eq:sizes}
    1 + \sum_{j=1}^m (|\sigma_j|-1) = n.
  \end{equation}
\end{remark}

\begin{example}
  The cycles $\{ (1\,4\,5),\ (1\,3),\ (2\,4) \}$ cannot be arranged
  into a factorization since the curves connecting $(1\,3)$ and
  $(2\,4)$ cannot be drawn inside a circle without intersecting.  The
  cycles $\{(1\,4\,5),\ (1\,2\,3),\ (3\,4) \}$ cannot be arranged into
  a factorization since the resulting circle drawing has a
  non-retractable loop $1, 4, 3, 1$.  The cycles $\{(1\,4\,5),\
  (1\,2),\ (2\,3) \}$ satisfy the conditions of
  Theorem~\ref{thm:struct_fact} and yield the factorization
  $(1\,4\,5)(1\,2)(2\,3)$.
\end{example}

\begin{remark}
  The result of drawing the cycles on the circle (after the circle has
  been erased) is a cactus, precisely the object that was used by
  Irving\cite{Irv06} to count factorizations.  We will prove
  Theorem~\ref{thm:struct_fact} in Section~\ref{sec:struct} by
  establishing a bijection between our trees and cacti drawn on a
  circle.  An informal explanation for the bijection is rather simple:
  while trees were obtained from the shuttle diagrams by shrinking
  vertical edges, the cacti are obtained by shrinking horizontal edges
  (provided the cycles are drawn as shown on
  Fig.~\ref{fig:shrinking}(b)).
\end{remark}

\section{Counting factorizations}
\label{sec:counting}

While Theorem~\ref{thm:gf_general} is a simple consequence of
Theorem~\ref{thm:bijection} and known counting results for trees
(\cite{ErdEth41}, \cite{Tut64} or \cite[Theorem 5.3.10]{StanleyV2}), we
provide a brief proof in order to introduce the methods used in the proof of
Theorem~\ref{thm:gf_heads_tails}.

\begin{proof}[Proof of Theorem~\ref{thm:gf_general}]
  We are going to enumerate the plane trees which have $\alpha_j$ vertices of
  degree $2j$.  The set of all such trees will be denoted by
  $\mathcal{T}_\balpha$, where $\balpha=(\alpha_2,\alpha_3,\ldots)$.

  \begin{figure}[t]
    \begin{center}
      \includegraphics[scale=0.8]{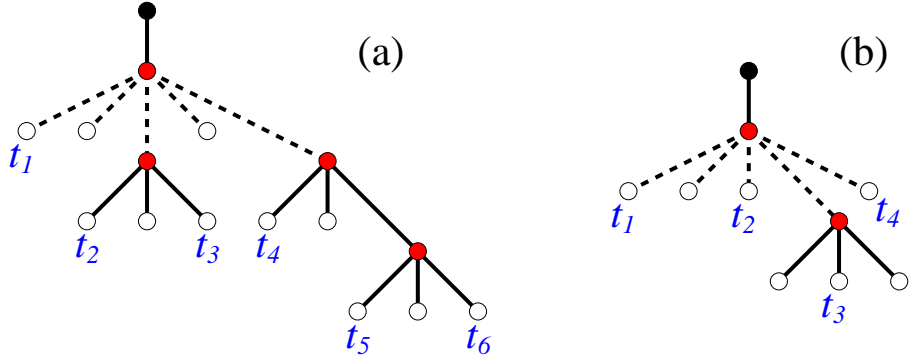}
    \end{center}
    \caption{(a) A tree with characteristic $\balpha=(3,1,0)$
      separates at the top internal vertex into five subtrees
      characterized by $\balpha_1=\balpha_2=\balpha_4=(0)$,
      $\balpha_3=(1,0)$ and $\balpha_5=(2,0)$. (b) A tree with
      characteristic $\balpha=(1,1,0)$.}
    \label{fig:recur1}
  \end{figure}

  To derive a recurrence relation for $|\mathcal{T}_\balpha|$ we break the
  tree at the top vertex adjacent to the root.  The top vertex has degree
  $2(d+1)$ for some $d\geq1$ and, when splitting the tree, it becomes the root
  of $2d+1$ subtrees $T_1,\dots,T_{2d+1}$, characterized by vectors
  $\balpha_1,\dots,\balpha_{2d+1}$ (some of them possibly empty).  Clearly
  $\balpha=\sum_{i=1}^{2d+1} \balpha_i +\be_d$ where $\be_d$ has $1$ in its
  $d$-th component\footnote{we remind the reader that the $k$-th component of
    vector $\balpha$ is $\alpha_{k+1}$} and zero elsewhere, representing the
  top vertex that was removed.  Figure \ref{fig:recur1}(a) shows a tree with
  characteristic $(3,1,\boldsymbol{0})$.  This tree splits at the top
  vertex, degree six ($d=2$), into five subtrees. The number of all possible
  trees with the top vertex of degree $2(d+1)$ is given by the number of
  combinations of subtrees, $\prod_{j=1}^{2d+1}|\mathcal{T}_{\balpha_j}|$,
  where $\sum_{j=1}^{2d+1} \balpha_j = \balpha-\be_d$.  Summing over the
  possible degrees of the top vertex establishes the recursion relation,
  \begin{equation}
    \label{eq:recur_N2}
    \widetilde{H}(\balpha) = |\mathcal{T}_\balpha|
    = \sum_{d\ge 1} \sum_{\balpha_1\cdots\balpha_{2d+1}}
    \prod_{j=1}^{2d+1}\widetilde{H}(\balpha)
    \delta_{\balpha_1+\ldots+\balpha_{2d+1},\,\balpha-\be_d} \ .
  \end{equation}
  Computing the generating function $\xi(\bx)$,
  equation~(\ref{eq:gen_fun_simple}), we recover (\ref{eq:gf_general}).

  To relate $\widetilde{H}(\balpha)$ to Catalan numbers (something
  important in applications, \cite{BerHarNov08}), we take
  $x_j=-r^{j-1}$, $j\ge 2$.  These particular values lead to
  \begin{equation*}
    x_2^{\alpha_2}x_3^{\alpha_3}\cdots=(-1)^{|\balpha|}r^{\langle
    \balpha\rangle}.
  \end{equation*}
  On the other hand, recursion (\ref{eq:gf_general}) implies that
  \begin{equation*}
    \widetilde{\xi}=1-r\widetilde{\xi}^3-r^2\widetilde{\xi}^5-\cdots,\qquad
    \mbox{ where }\quad
    \widetilde{\xi}(r)=\xi(-r, -r^2, \ldots).
  \end{equation*}
  The right-hand side is almost a geometric series; we multiply the
  equation by $1-r\widetilde{\xi}^2$ to arrive at
  \begin{equation*}
    r\widetilde{\xi}^2 + \widetilde{\xi} - 1 = 0.
  \end{equation*}
  This can be solved for $\widetilde{\xi}$ and results in the well known
  generating function of $(-1)^nc_n$.
\end{proof}

A recurrence relation for $\widetilde{H}_{\bh,\bt}(\balpha)$ can be established
in a similar manner.

\begin{proof}[Proof of Theorem~\ref{thm:gf_heads_tails}]
  We recap that we are counting the factorizations with a given number of
  heads and tails.  On a tree, a tail corresponds to a vertex of degree $2j$
  which has $j$ free $t$-labeled edges attached to it.  For example, on
  Fig.~\ref{fig:recur1}, there are $t=2$ tails.  Similarly a head is a degree
  $2j$ vertex with $j$ free $h$-labeled edges attached (we omitted $h$ labels
  from Fig.~\ref{fig:recur1} and other figures to avoid clutter).  Note that
  the top vertex can be both a tail and a head, although not simultaneously,
  at least for trees with more than one vertex.  The root counts as being
  $h$-labeled and is always free.  For example, the tree on
  Fig.~\ref{fig:recur1}(a) has the top vertex as its only head, $h=1$.  We
  also introduce a variable $h'$ counting all heads excluding the top vertex.
  We will refer to it as the {\em reduced head count\/} and for the tree on
  Fig.~\ref{fig:recur1}(a) it is $h'=h-1=0$, while for the tree on
  Fig.~\ref{fig:recur1}(b) $h'=h=1$.  We will first derive a recursion
  counting the trees with a given reduced head count and from there obtain the
  number of trees with full head count.

  The tail and (reduced) head count are further specialized to count the
  number of heads and tails of a certain degree.  Thus, in general, $\bh$,
  $\bh'$ and $\bt$ are infinite vectors with finitely many nonzero components.
  Let $\phi$ be a partial generating function with respect to the tail and
  reduced head count
  \begin{equation*}
    \phi(\balpha,\bv,\bu) =
    \sum_{\bh'=(0,0,\ldots)}^{\balpha} \sum_{\bt=(0,0,\ldots)}^{\balpha}
    \widetilde{H}_{\bh',\bt}(\balpha)
    u_2^{h_2'} v_2^{t_2}
    u_3^{h_3'} v_3^{t_3} \cdots, \qquad \phi(\bo,\bv,\bu) = 1.
  \end{equation*}

  To establish the recursion relation we again consider breaking the tree into
  subtrees $T_1, \dots , T_{2d+1}$ at the top vertex of degree $2(d+1)$,
  numbering the subtrees left to right.  As before, the subtrees are
  characterized by vectors $\balpha_1,\dots,\balpha_{2d+1}$.  We introduce a
  special notation for the sum of odd-indexed vectors and for the sum of
  even-indexed ones,
  \begin{equation}
    \label{eq:alpha_odd_even}
    \balpha^o = \sum_{j=0}^d \balpha_{2j+1}
    \qquad \balpha^e = \sum_{j=1}^d \balpha_{2j}.
  \end{equation}

  The reduced head count $\bh' = (h_2,h_3,\ldots)$ of the full tree
  can be obtained by summing the appropriate counts for the subtrees,
  namely
  \begin{equation}
    \label{eq:reduced_head_cnt}
    \bh'(T)=\bh'(T_1) + \sum_{j=1}^{d} \big(\bt(T_{2j})+ \bh'(T_{2j+1})\big).
  \end{equation}
  Note that for the even-numbered subtrees, we need to add the number
  of tails rather than heads.  This corresponds to a change in the
  labeling of the leaves on the subtrees with even index.  On subtrees
  with odd index the first (leftmost) leaf is always $t$-labeled,
  while the first leaf of an even-numbered subtree is $h$-labeled,
  see Fig.~\ref{fig:recur1}(b) for an example.

  For the tail count of the complete tree, the procedure is analogous, with
  the addition of the possible contribution of the top vertex.  The top vertex
  is a tail if all the odd subtrees are empty, i.e. $\balpha_{2j+1} = \bo$,
  $j=0,\ldots,d$.  Figure \ref{fig:recur1}(b) shows a tree where the top
  vertex is a tail.  Therefore,
  \begin{equation*}
    \bt(T) = \bt(T_1) + \sum_{j=1}^{d} \big(\bh'(T_{2j}) + \bt(T_{2j+1})\big)
    + \delta_{\balpha^o, \bo} \be_d.
  \end{equation*}

  Consequently $\phi(\balpha,\bv,\bu)$ is expressed in terms of functions
  $\phi(\balpha_j,\bv,\bu)$ generated by the subtrees,
  \begin{multline}
    \label{eq:phi}
    \phi(\balpha,\bv,\bu) = \sum_{d\ge 1} \sum_{\balpha_1\cdots\balpha_{2d+1}}
    \phi(\balpha_1,\bv,\bu)
    \prod_{j=1}^{d}\phi(\balpha_{2j},\bu,\bv)\phi(\balpha_{2j+1},\bv,\bu)
    \\ \times
    \left(1-(1-v_{d+1})\delta_{\balpha^o,\bo}\right) \,
    \delta_{\balpha^o+\balpha^e,\,\balpha-\be_d}.
  \end{multline}
  It is important to observe that the functions $\phi$ with
  even-indexed vectors $\balpha_{2j}$ have their arguments $\bu$ and
  $\bv$ switched around.  Calculating the generating function
  \begin{equation*}
    f(\bx,\bu,\bv) = \sum_{\balpha} \phi(\balpha,\bv,\bu) x_2^{\alpha_2}
    x_3^{\alpha_3} \cdots \, ,
  \end{equation*}
  we recover recurrence relation (\ref{eq:reduced_rec}).

  The complete head count can be obtained from the appropriate counts
  for the subtrees in a slight variation of
  (\ref{eq:reduced_head_cnt}),
  \begin{equation}
    \label{eq:full_head_cnt}
    \bh(T) = \bh'(T_1)
    + \sum_{j=1}^{d} \big(\bt(T_{2j}) + \bh'(T_{2j+1})\big)
    + \delta_{\balpha^e, \bo}\, \be_d,
  \end{equation}
  where $\balpha^e$ was defined in equation~(\ref{eq:alpha_odd_even}).

  The partial generating function with respect to the full head count
  is then
  \begin{multline*}
    \psi(\balpha,\bv,\bu) =
    \sum_{\bh=(0,0,\ldots)}^{\balpha} \sum_{\bt=(0,0,\ldots)}^{\balpha}
    \widetilde{H}_{\bh,\bt}(\balpha)
    u_2^{h_2} v_2^{t_2} u_3^{h_3} v_3^{t_3} \cdots \\
    = \sum_{d\ge 1} \sum_{\balpha_1\cdots\balpha_{2d+1}}
    \phi(\balpha_1,\bv,\bu)
    \prod_{j=1}^{d}\phi(\balpha_{2j},\bu,\bv)\phi(\balpha_{2j+1},\bv,\bu)
    \\ \times
    \left[1-(1-v_{d+1})\delta_{\balpha^o,\bo}
      - (1-u_{d+1})\delta_{\balpha^e,\bo}\right] \,
    \delta_{\balpha^o+\balpha^e,\,\balpha-\be_d}
  \end{multline*}
  Opening the square brackets, using the recursion (\ref{eq:phi}) for $\phi$, and the fact that $\balpha^e=\bo$ implies
  $\phi(\balpha_{2j},\bu,\bv) = 1$, we obtain
  \begin{multline*}
    \psi(\balpha,\bv,\bu)
    = \phi(\balpha,\bv,\bu)
    - \sum_{d\ge 1} \sum_{\balpha_1\cdots\balpha_{2d+1}}
    \prod_{j=0}^{d}\phi(\balpha_{2j+1},\bv,\bu)
    (1-u_{d+1})\, \delta_{\balpha^o,\,\balpha-\be_d}.
  \end{multline*}

  Calculating the full generating function $g(\bx,\bu,\bv)$ we obtain
  \begin{equation*}
    g(\bx,\bv,\bu) = f - \sum_{d\geq1} x_{d+1} (1-u_{d+1}) f^{d+1},
  \end{equation*}
  which is the same as (\ref{eq:rec_full_unsym}) after the substitution
  $n=d+1$.  We now transform this relation to form (\ref{eq:rec_full_sym}),
  which confirms that, in contrast to $f(\bx,\bv,\bu)$, the generating function
  $g(\bx,\bv,\bu)$ is symmetric with respect to the exchange of $\bu$ and
  $\bv$.  We exchange $\bu$ and $\bv$ in (\ref{eq:reduced_rec}) to obtain a
  recursion for $\hat{f}$,
  \begin{equation*}
    \hat{f} = 1
    + \sum_{n\geq2} x_n \left(\hat{f}^n - 1 + u_n\right) f^{n-1},
  \end{equation*}
  multiply it by $f$ and rearrange,
  \begin{equation*}
    f\hat{f} = f - \sum_{n\geq2} x_n(1-u_n) f^n
    + \sum_{n\geq2} x_n \hat{f}^n f^n,
  \end{equation*}
  from which (\ref{eq:rec_full_sym}) immediately follows.
\end{proof}

\section{Trees and cacti}
\label{sec:struct}

As seen in Section~\ref{sec:viz} and Theorem~\ref{thm:bijection}, a
factorization can be visualized as a tree with vertices representing
factors.  An alternative visualization involves drawing cycles as
(curvilinear) polygons inscribed in a circle, leading to an inscribed
cactus.  As we mentioned earlier, the two models can be viewed as two
ways of removing redundant information from the shuttle diagrams of
Section~\ref{sec:viz}.  The trees are obtained by shrinking vertical
edges (shuttles), while cacti are obtained by shrinking horizontal
edges.

In this Section we will give a direct mapping between trees and cacti
inscribed on a circle.  Namely, we will show that cacti satisfying
conditions \ref{itm:cover_all}-~\ref{itm:simply_conn} of
Theorem~\ref{thm:struct_fact} are in one-to-one correspondence with
the rooted plane trees of Theorem~\ref{thm:bijection}.  Then the
statement of Theorem~\ref{thm:struct_fact} follows directly from
Theorem~\ref{thm:bijection} and Lemma~\ref{lem:increasing}.  We will
keep our exposition slightly informal since a formal proof of a
similar result for transpositions is available, for example, in
\cite{EdeSch_mtamkik62} (see also an exposition in Section 4.5 of
\cite{Berge_principles}).

\subsection{From an inscribed cactus to a tree}

\begin{figure}[t]
  \centering
  \includegraphics[scale=0.6]{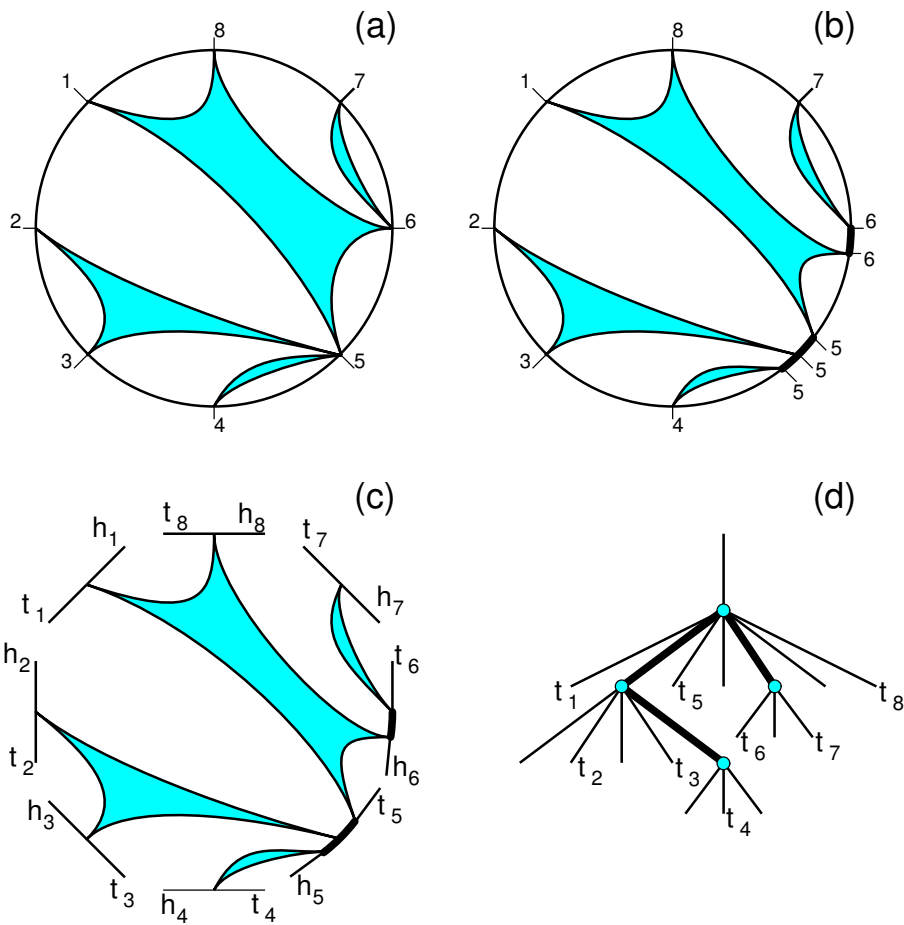}
  \caption{Turning a cactus into a tree: (a) original cactus inscribed
  in a circle, (b) adding edges (bold) between touching corners of
  polygons, (c) cutting circle between labels $j$ and $j+1$ and (d)
  shrinking polygons to produce internal vertices of the tree.}
  \label{fig:cactus_to_tree}
\end{figure}

To transform an inscribed cactus into a tree we start by separating
on the circle the touching corners of polygons.  This creates more
edges (shown in bold lines on Fig.~\ref{fig:cactus_to_tree}).  The new
vertices on the circle will receive the same label as the original
vertex.  Thus, if $k$ polygons touched at vertex $j$ there are now $k$
vertices labeled $j$ in consecutive positions around the circle.  The
new edges will become the internal edges (i.e. not ending in a degree
one vertex) of the resulting graph.  We observe that after this step
no two polygons have any points in common.  Also, each corner of each
polygon is connected to exactly two edges.

In the next step, every arc or the circle connecting vertices $j$ and
$j+1$ (by $n+1$ we understand $1$) is cut in half.  The new vertices receive
labels $t_j$ and $h_{j+1}$, so that $t_j$ is connected to vertex $j$
and $h_{j+1}$ is connected to $j+1$, see an example on
Fig.~\ref{fig:cactus_to_tree}, part (c).

Finally, all polygons are shrunk to form vertices,
Fig.~\ref{fig:cactus_to_tree}(d).  Since the original cactus was
simply connected, the result is a tree with (by construction) the
correct vertex degrees to satisfy Theorem~\ref{thm:bijection}.  Thus
it corresponds to a minimal factorization of the full $n$-cycle.

\subsection{From a tree to an inscribed cactus}

\begin{figure}[t]
  \centering
  \includegraphics[scale=0.6]{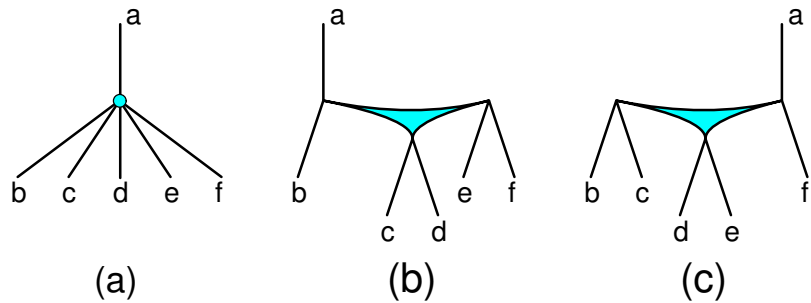}
  \caption{Expanding a tree vertex into a polygon.  A vertex of degree
    $2m$ becomes a polygon with $m$ corners.  Each corner is connected
    to a pair of edges.  If the vertex has $t$-type, we group the
    edges in pairs according to figure (b).  Vertices of $h$-type
    become polygons depicted on figure (c).}
  \label{fig:tree_to_cactus}
\end{figure}

To make a cactus out of a tree we essentially reverse the procedure
outlined in the previous section.  The only difficult point is the
``inflation'' of internal vertices into polygons.  Each vertex of
degree $2m$ will become a polygon with $m$ corners and two edges
attached to each corner.  Since the circular ordering of the edges
around the polygon is determined by the ordering around the tree
vertex, there are only two possibilities to group edges into pairs
(depicted in Fig.~\ref{fig:tree_to_cactus}).

To decide which vertex becomes which type of polygon, we determine the
``type'' of a vertex $v$.  Imaging a counter-clockwise walk starting
at the leaf $h_1$ and following the exterior of the tree.  If the last
leaf visited by the walk prior to coming to $v$ \emph{for the first time}
was a $t$-leaf, the vertex $v$ is of type $h$.  Otherwise, it is of
type $t$.  Note that this definition works for leafs as well as
internal vertices, and produces for them the ``correct'' type.  The
type is also preserved in building a tree from the top down.  Having
determined the type of a vertex, the vertices of type $t$ become
polygons of the type depicted on Fig.~\ref{fig:cactus_to_tree}(b) and
vertices of type $h$ become polygons similar to
Fig.~\ref{fig:cactus_to_tree}(c).

From this point, we merge pairs of vertices $t_j$ and $h_{j+1}$ to
form the circle and then shrink the internal edges, obtaining an
inscribed cactus.

\section{Conclusions and outlook}

The simple pictorial bijection introduced in
Theorem~\ref{thm:bijection} has allowed us to perform an in-depth
analysis of the set of inequivalent minimal factorizations of the
$n$-cycle.  The next logical step is to apply similar ideas to
inequivalent minimal transitive factorizations of a general
permutation.  Our preliminary explorations showed that the ideas of
the present manuscript provide a method for deriving a recursion for
the generating function for any finite $m$, where $m$ is the number of
cycles in the cycle representation of the target permutation ($m=1$
corresponds to an $n$-cycle).  We have also found \cite{BerKui_prep10}
that this question is directly applicable in computing non-linear
moments of transmission probability through a chaotic quantum system.
However, for the above application some information on the number of
tails and heads in a factorization is again required.

\section*{Acknowledgment}
The authors acknowledge the discussions they had with P.~Lima-Filho
and F.~Sottile and thank them for making useful suggestions.  The
authors are extremely grateful to J.~Irving for sending us a copy of
manuscript \cite{Spr96}.

\bibliographystyle{ieeetr} \bibliography{factor_cycle}

\end{document}